\documentclass{nature}

\usepackage{graphicx}

\usepackage{epstopdf}
\usepackage{amsbsy}
\usepackage{appendix}
\usepackage{ulem}
\usepackage{color}
\usepackage{gensymb}
\usepackage{amsmath}
\usepackage{amsthm}
\usepackage{soul}

\usepackage[super,comma,sort&compress]{natbib}
\newcommand{\beq}{\begin{equation}}
\newcommand{\eeq}{\end{equation}}
\newcommand{\beann}{\begin{eqnarray*}}
\newcommand{\eeann}{\end{eqnarray*}}

\newtheorem{lemma}{Lemma}
\theoremstyle{definition}
\newtheorem{definition}{Definition}

\definecolor{ao(english)}{rgb}{0.0, 0.5, 0.0}

\title{Many-Body Fermions and Riemann Hypothesis}

\author{Xindong Wang and Alex Shulman}

\begin{document}
\clearpage\maketitle
\thispagestyle{empty}

\begin{affiliations}
\item  Sophyics Technology, LLC
\end{affiliations}

\date{\today}

\begin{abstract}
We study the algebraic structure of the eigenvalues of a Hamiltonian that corresponds to a many-body fermionic system. As the Hamiltonian is quadratic in fermion creation and/or annihilation operators, the system is exactly integrable and the complete single fermion excitation energy spectrum is constructed using the non-interacting fermions that are eigenstates of the quadratic matrix related to the system Hamiltonian. Connection to the Riemann Hypothesis is discussed.
\end{abstract}

\clearpage
\setcounter{page}{1}

Riemann Hypothesis has long been conjectured to be related to the eigenvalues of a Hamiltonian\cite{BerryKeating} since Hilbert in early twentieth century. In this paper, we show that the eigenvalues of an anti-symmetric real matrix that arises from the off-diagonal paring matrix elements of a many-body fermionic Hamiltonian seems to provide the necessary link between the Berry-Keating\cite{BerryKeating,BBM} Conjecture and the final proof of Riemann Hypothesis. This work points to the importance of Riemann Hypothesis to the understanding of intricate quantum entanglement of a many body system.

\section{Off-diagonal ordering in many-body fermionic system}
We study the following 1-dimensional spin-half many-body fermionic Hamiltonian
\begin{equation} \label{water effective hamiltonian}
\begin{aligned}
    \hat{H} =& \sum_{i,\sigma} \frac{\sigma}{2}\{ \hat{p}^\dagger_{i\sigma} \hat{p}_{i\sigma}  -  \hat{h}^\dagger_{i\sigma} \hat{h}_{i\sigma} \}  -{\big \{}\sum_{i>i',\sigma} t({ i-i'}) \hat{p}^\dagger_{i\sigma} \hat{h}^\dagger_{i'-\sigma} + h.c. {\big \}}\\=& \begin{bmatrix}
    \hat{\bf \xi}^\dagger_\uparrow & \hat{\bf \xi}_\downarrow^\dagger
    \end{bmatrix} \begin{bmatrix}T_\uparrow & 0\\  0 &T_\downarrow\end{bmatrix} \begin{bmatrix}\hat{\bf \xi}_\uparrow \\ \hat{\bf \xi}_\downarrow
    \end{bmatrix}  \   \ i\in\{1,2,...,N\}, \   \ \sigma\in\{\uparrow, \downarrow\}, \    \ N\geq 2
\end{aligned}
\end{equation}
and we further assume $t(i+N) = t(i)$, i.e., the system is a closed loop. 
\begin{equation}
    \hat{\bf \xi}^\dagger_\sigma = \begin{bmatrix} 
    \hat{p}^\dagger_{1\sigma} &  \hat{p}^\dagger_{2\sigma} & ... & \hat{p}^\dagger_{N\sigma}  & \hat{h}_{1-\sigma}& \hat{h}_{2-\sigma} ... &\hat{h}_{N-\sigma}
    \end{bmatrix}  
\end{equation}
\begin{equation} \label{T matrix}
    T_\sigma= \sigma \begin{bmatrix}\frac{1}{2}I_{N} & \sigma \Delta\\  \sigma \Delta^\dagger &-\frac{1}{2}I_{N}\end{bmatrix}
\end{equation}
and
\begin{equation} \label{Delta}
    \Delta=\frac{1}{2}\begin{bmatrix}
    0 & t(1) & t(2) & ... & t(N-1) \\
    -t(1) & 0 & t(1) & ... & t(N-2) \\
    ... \\
    -t(N-1) & -t(N-2) & -t(N-3) & ... &0
    \end{bmatrix}
\end{equation}
is an anti-symmetric matrix, due to the anti-commutative relation of the fermionic operators.

The total charge operator for this system is defined as
\begin{equation} \label{charge operator}
\hat{N}_c = \sum_{i\sigma} (\hat{p}^\dagger_{i\sigma} \hat{p}_{i\sigma} - \hat{h}^\dagger_{i\sigma} \hat{h}_{i\sigma})
\end{equation}
and the total spin operator
\begin{equation}
    \hat{\Sigma}_3 = \sum_{i\sigma} \sigma(\hat{p}^\dagger_{i\sigma} \hat{p}_{i\sigma} + \hat{h}^\dagger_{i\sigma} \hat{h}_{i\sigma}) = \sigma \hat{N}_{c\sigma}
\end{equation}
where
\begin{equation}
\hat{N}_c\sigma = \sum_{i} (\hat{p}^\dagger_{i\sigma} \hat{p}_{i\sigma} + \hat{h}^\dagger_{i-\sigma} \hat{h}_{i-\sigma})
\end{equation}
And one can show that the total charge operator commutes with the Hamiltonian \eqref{water effective hamiltonian} 
\begin{equation}
    [\hat{H}, \hat{N}_c] = 0 ,\  \ [\hat{H}, \hat{\Sigma}_3] = 0
\end{equation}
We will focus on the case where the total charge as well as the total spin of the system is zero, i.e., zero chemical potential and zero external magnetic field.

Since the two spin channels are completely decoupled and degenerate, we will only need to discuss the energy spectrum of $T_\uparrow$ below.

For a Hamiltonian of quadratic form, it can be exactly diagonalized in the subspace of zero charge as
\begin{equation}
    \hat{H} = \sum_{(n\sigma)\in\{(n\sigma)|\varepsilon_n \le 0\}}  \varepsilon_n \{|Vac_h\rangle\langle Vac_h|+|Vac_p\rangle\langle Vac_p|\} + \sum_{n\sigma\omega\in\{h,p\}} |\varepsilon_n| \hat{\gamma}_{n\sigma\omega}^\dagger\hat{\gamma}_{n\sigma\omega} 
\end{equation}
where 
\begin{equation}
\begin{aligned}
    &\hat{\gamma}_{n\sigma p} = \sum_{i} u^\sigma_{n,i} \hat{p}_{i\sigma} + \sum_j v^\sigma_{n, j}\hat{h}_{j-\sigma}^\dagger \\
    &\hat{\gamma}_{n\sigma h} = \sum_{i} u^\sigma_{n,i} \hat{h}_{i-\sigma} + \sum_j v^\sigma_{n, j}\hat{p}_{j\sigma}^\dagger
\end{aligned}
\end{equation}
and the coefficients and $\varepsilon_n$ are eigenvectors and eigenvalues of the following Hermitian matrix T defined in Eq.\eqref{T matrix}
\begin{equation} \label{single fermion excitation}
   T \begin{bmatrix}
    {\bf u}_n \\ {\bf v}_n 
    \end{bmatrix} = \varepsilon_n \begin{bmatrix}
    {\bf u}_n \\ {\bf v}_n 
    \end{bmatrix}
\end{equation}
Note that we have explicitly retained the two time reversal symmetry related degenerate vacuum states $|Vac_p\rangle$ and $|Vac_h\rangle$, representing the two degenerate ground states of filled Fermi sea of p-fermions or h-fermions, and $\hat{\gamma}_p, \hat{\gamma}^\dagger_p$ and $\hat{\gamma}_h, \hat{\gamma}_h^\dagger$ represent the Majorana fermions corresponding to their respective vacuum states.

Note that the two time reversal symmetry related sets of solutions can be considered decoupled to each other at the thermodynamic limit, since the only common eigenstate for each set is the absolute empty vacuum where all the filled Fermi sea fermions in the 2 vacuum states are all excited. Thus in the thermodynamic limit, we can consider the two sets of solutions two time-reversal symmetry related universe.

Expand explicitly the Eq.\eqref{single fermion excitation}, we have
\begin{equation}
\begin{aligned}
    &(1/2-\varepsilon_n) I_N {\bf u}_n + \Delta {\bf v}_n = 0 \\
    &\Delta^\dagger {\bf u}_n + (-1/2 -\varepsilon_n) I_N {\bf v}_n = 0
\end{aligned}
\end{equation}
which leads to
\begin{equation}
    \{(\varepsilon_n - 1/2)(\varepsilon_n + 1/2) I_N - \Delta^\dagger \Delta\} {\bf v} = 0
\end{equation}
That is $\varepsilon_n$ are roots of the following polynomial
\begin{equation}
    \mathcal{P}(z)=det\begin{bmatrix}
    (z^2-(1/2)^2) I_N - \Delta^\dagger \Delta
    \end{bmatrix}
\end{equation}

The eigenvalues $\varepsilon_n$ are thus
\begin{equation}
    \varepsilon_n = \sqrt{(1/2)^2+t_n^2}
\end{equation}
where $t_n^2$ are singular values of $\Delta^\dagger\Delta$

\section{Algebraic structure of an anti-symmetric matrix}
To make the connection to Riemann Hypothesis, we first develop the mathematical theory for the algebraic structure of eigenvalues of anti-symmetric matrices.
\begin{lemma} \label{lemma}
For any given anti-symmetric matrix $\Delta$, it can be unitarily diagonalized. And when $\Delta$ is anti-symmetric real, all its eigenvalues are imaginary.
\end{lemma}
\begin{proof}
Since $\Delta = \Delta_r + i\Delta_i$  is anti-symmetric, $\Delta_r, \Delta_i$ are real anti-symmetric, 
\begin{equation}
    \Delta\Delta^\dagger = \Delta^\dagger\Delta = -\Delta_r^2 + \Delta_i^2 
\end{equation}
Thus both $\Delta$ and $\Delta^\dagger$ are normal matrices, i.e., they can be unitarily diagonalized
\begin{equation} \label{central theorem}
    \Delta = U D U^\dagger, \   \
    \Delta^\dagger = U D^* U^\dagger
\end{equation}
where D is a diagonal matrix.

When $\Delta$ is anti-symmetric real matrix, $i\Delta$ is a Hermitian matrix, thus 
it can be diagonalized with all eigenvalues being real, i.e., $iD$ is a real diagonal matrix. This completes the proof.
\end{proof}

Next we show that following Lemma concerning the rank of an anti-symmetric matrix
\begin{lemma} \label{rank lemma}
If M is an anti-symmetric matrix of size $N$, denote its rank as $r_M$, then
$r_M$ is an even number and all non-zero eigenvalues of $N$ come in pairs of $\pm z_i$.
\end{lemma}
\begin{proof}
This is because every eigenvalue of $M$ is also an eigenvalue of $M^T$ and $M^T=-M$, so if $\lambda$ is an eigenvalue of $M$, then $-\lambda$ is also an eigenvalue. Thus, all non-zero eigenvalues of $M$ come in pairs. This completes the proof.
\end{proof}
We note that both Lemmas have been established in the literature \cite{Youla1961}.

\begin{lemma} \label{gauge transformation theorem}
For any anti-symmetric matrix $M$, there exists an anti-symmetric real matrix $\tilde{M}$, such that $MM^\dagger$ and $\tilde{M}\tilde{M}^\dagger$ have the same eigenvalues and they are related by a unitary transformation $\mathcal{U}$ or an anti-unitary transformation $\mathcal{A}=\mathcal{K U}$, i.e.,
\begin{equation}
    M = \mathcal{U} \tilde{M} \mathcal{U}^\dagger = \mathcal{UK} \tilde{M}  \mathcal{K}\mathcal{U}^\dagger
\end{equation}
where $\mathcal{K}$ is the complex conjugate operator.
\end{lemma}
\begin{proof}
$M=UDU^\dagger$ by Lemma \ref{lemma}, and due to Lemma \ref{rank lemma}, $D$ can be arranged as the following blocks of pairs
\begin{equation}
    D=\begin{bmatrix}
    \begin{pmatrix}\lambda_1 & 0\\0&-\lambda_1\end{pmatrix} & 0& ... & 0  & 0\\
    0 & \begin{pmatrix}\lambda_2 & 0\\0&-\lambda_2\end{pmatrix}& ... & 0 & 0\\
    ...\\
    0&0&...&\begin{pmatrix}\lambda_{r_M/2} & 0\\0&-\lambda_{r_M/2}\end{pmatrix}& 0\\
    0&0&0&...&0
    \end{bmatrix}
\end{equation}
Define the following phase factors $\varphi_i$
\begin{equation}\label{gauge field}
    \lambda_i = i \varepsilon_i e^{i \varphi_i}
\end{equation}
where $\varepsilon_i = |\lambda_i|$.
Thus we have
\begin{equation} \label{gauge tr}
    D=\Phi(\{ \varphi_i \})E \Phi^\dagger(\{ \varphi_i \})
\end{equation}
where
\begin{equation}
    \Phi(\{\varphi_i\}) = \begin{bmatrix}
    e^{i \frac{ \varphi_1}{2}}I_2&0& ... & 0 & 0\\
    0&e^{i \frac{ \varphi_2}{2}}I_2&...&0&0
    ...\\
    0&0&...&e^{i \frac{ \varphi_{r_M/2}}{2}}I_2& 0\\
    0&0&0&...&0
    \end{bmatrix}
\end{equation}
is a unitary matrix and $E$ is 
\begin{equation}
    E=\begin{bmatrix}
    \begin{pmatrix}i\varepsilon_1 & 0\\0&-i\varepsilon_1\end{pmatrix} & 0& ... & 0  & 0\\
    0 & \begin{pmatrix}i\varepsilon_2 & 0\\0&-i\varepsilon_2\end{pmatrix}& ... & 0 & 0\\
    ...\\
    0&0&...&\begin{pmatrix}i\varepsilon_{r_M/2} & 0\\0&-i\varepsilon_{r_M/2}\end{pmatrix}& 0\\
    0&0&0&...&0
    \end{bmatrix}
\end{equation}
And we have
\begin{equation*}
    \begin{pmatrix}i\varepsilon_i & 0\\0&-i\varepsilon_i\end{pmatrix} = v_i\begin{pmatrix}0&\varepsilon_i  \\-\varepsilon_i&0\end{pmatrix}v_i^\dagger
\end{equation*}
where $v_i v_i^\dagger = 1$, as
\begin{equation*}
    v_i = \begin{pmatrix} \frac{1}{\sqrt{2}} & \frac{-i}{\sqrt{2}} \\ \frac{1}{\sqrt{2}} & \frac{i}{\sqrt{2}}\end{pmatrix}\quad v^\dagger_i = \begin{pmatrix} \frac{1}{\sqrt{2}} & \frac{1}{\sqrt{2}} \\ \frac{i}{\sqrt{2}} & \frac{-i}{\sqrt{2}}\end{pmatrix}
\end{equation*}
Thus we have
\begin{equation}
    D=\Phi(\{\varphi_i\}) V\cdot \mathcal{E} \cdot V^\dagger \Phi^\dagger(\{\varphi_i\})
\end{equation}
where $\mathcal{E}$ is an anti-symmetric real matrix and 
$\Phi(\{\varphi_i\})V$ is a unitary matrix
since product of two unitary matrix is still a unitary matrix.
Thus
\begin{equation*}
    MM^\dagger= U DD^\dagger U^\dagger = U\Phi(\{\varphi_i\})V \mathcal{E}\mathcal{E}^\dagger V^\dagger \Phi^\dagger(\{\varphi_i\})U^\dagger = {\big (}\mathcal{U} \mathcal{E} \mathcal{U}^\dagger {\big )(}\mathcal{U} \mathcal{E}^\dagger \mathcal{U}^\dagger {\big )}
\end{equation*}
where $\mathcal{U}$
\begin{equation}
    \mathcal{U} = U\Phi(\{\varphi_i\})V
\end{equation}
is unitary as products of unitary matrices are unitary.
And for the case of anti-unitary transformation, we observe that both  $\mathcal{E} $ and $\mathcal{E}^\dagger$ are real, we have
\begin{equation*}
    \mathcal{K}\mathcal{E}\mathcal{K} = \mathcal{E} \quad \mathcal{K}\mathcal{E}^\dagger\mathcal{K} = \mathcal{E}^\dagger
\end{equation*}
\end{proof}

Thus, any non-degenerate anti-symmetric real matrix defines an equivalence relationship per Lemma \ref{gauge transformation theorem}, if $MM^\dagger$ have the same set of pairs of eigenvalues. All members of the equivalence class are related by a unitary or an anti-unitary transformation.

\theoremstyle{definition}
\begin{definition}\label{chracteristic}
The real anti-symmetry matrix of an equivalent class of anti-symmetric matrices is called the characteristic of the class. We use $\mathcal{C}(M)$ to denote the characteristic of an anti-symmetric matrix.
\end{definition}

Intuitively, since the anti-symmetric matrices arise from the off-diagonal paring block of a many-body fermionic system, the topological effect of the anti-symmetric matrix is intricately related to the quantum entanglement of a many-body fermionic system and any periodicity in the off-diagonal matrix elements in real space will imply some harmonic resonances in the eigenvalues of the original Hermitian Hamiltonian, but in a way through the imaginary eigenvalues of the characteristics of the off-diagonal anti-symmetric matrix.

\section{Connection to Riemann Hypothesis}
Next we solve for the eigenvalues of the following anti-symmetric real matrix for a given $p$, where $p$ is a prime number that corresponds to the period of the gauge field. 

We will solve the eigenvalue problem with each $k \in[0,2/p]$. 
Let the anti-symmetric matrix $\Delta_N(p,k)$ be given explicitly as 
\begin{equation} \label{Delta}
    \Delta_{N}(p,k)=\frac{1}{2}\begin{bmatrix}
    0 & t_1(p,k) & t_2(p,k) & ... & t_{(N-1)}(p,k) \\
    -t_1(p,k) & 0 & t_1(p,k) & ... & t_{(N-2)}(p,k) \\
    ... \\
    -t_{(N-1)}(p,k) & -t_{(N-2)}(p,k) & -t_{(N-3)}(p,k) & ... &0
    \end{bmatrix}
\end{equation}

and $t_l(p,k), k\in[0,2/p]$ is 
\begin{equation}
    t_l(p,k) =p\cdot l \int_{-1/2}^{1/2} d q\cdot q\cdot  sin(2\pi (q+k)\cdot  l ) =  \frac{(-1)^l}{\pi }\cdot cos(2\pi l \cdot k),\quad l\in \{1,2,...,N-1\}
\end{equation}

And the matrix element for $\Delta_N(p,k)$ is anti-symmetric real:
\begin{equation}
    \Delta_{l,l'}(p,k) = - \Delta_{l',l}(p,k)
\end{equation}

Once the eigenvalues $\lambda_n(p,k)=i\varepsilon_n(p,k)$ are solved for all $k\in[0,2/p]$, then the following spectral function can be calculated
\begin{equation}
    \mathcal{G}_N (p,z)= \frac{p}{2N}
    \sum_n\int_0^{2/p} dk \frac{1}{z-\lambda_n(p,k)}
     =  \int_{-\infty}^{\infty} d\varepsilon \cdot \rho_N(p,\varepsilon)
     \frac{1}{z-i\varepsilon} 
\end{equation}
where the normalized density of state $\rho_N(p,\varepsilon)$ is given by
\begin{equation} \label{density of states}
    \rho_N(p,\varepsilon) = \frac{2}{Np} \sum_n \int_0^{p/2} dk \delta(\varepsilon - \varepsilon_n(p,k)) 
\end{equation}
and it has the following sum rule
\begin{equation}
    \int_{-\infty}^{\infty} d\varepsilon \cdot \rho_N(p, \varepsilon) = 1
\end{equation}

Note that $\mathcal{G}_N(p,z)$ contains the periodicity of $p$ of the underlying gauge field, thus poles of this function are harmonic resonances, that is when $N\rightarrow\infty$, we have
the density of states $\rho(\varepsilon)$ defined in Eq.\eqref{density of states} diverges at those resonance frequencies.

The poles of $\mathcal{G}_N(p,z)$ are all expected to be imaginary since all $\lambda_n(p, k)$ are imaginary. Thus poles of the following function
\begin{equation} \label{propagator}
    \mathcal{G}_N (z) = \prod_{p<N} \mathcal{G}_N(p,z)
\end{equation}
where $p$ are prime numbers, are all imaginary.

Define the following function
\begin{equation}
    \mathcal{G}(z) =\lim_{N\rightarrow \infty}  \mathcal{G}_N(z)
\end{equation}

We conjecture that poles of $ \mathcal{G}(z)$ are the imaginary part of the non-trivial zeros of the Riemann zeta function, to within a scaling factor.

Note that in the $N\rightarrow\infty$ limit, the non-Hermitian anti-symmetric matrix $\Delta_N(p,k)$ defined above approaches the operator $\hat{x}\hat{p}$, with additional topological phase factor picked up by the momentum operator. The choice of the particular form of the hopping matrix element is inspired by Berry-Keating \cite{BerryKeating} conjecture, especially the work of Bender, Brody, and Müller \cite{BBM}.

{\bf Acknowledgement} This work is supported by Sophyics Technology, LLC.

\bibliography{./refs.bib}

\end{document}